\documentclass[11pt]{article}

\usepackage{amsmath}
\usepackage{amssymb}
\usepackage{amsfonts}
\usepackage{amsthm}
\usepackage{latexsym}

\def\odd{\mathrm{odd}}

\newtheorem{thm}{Theorem}
\newtheorem{theorem}[thm]{Theorem}

\newtheorem{lemma}[thm]{Lemma}
\theoremstyle{definition}
\newtheorem{example}[thm]{Example}
\newtheorem*{remark}{Remark}

\title{Simple greedy $2$-approximation algorithm for the maximum genus of a graph} 

\author{Michal Kotrb\v{c}\'ik\thanks{Department of Computer Science, Faculty of Informatics, Masaryk University, Botanick\'a 68a, 602 00 Brno, Czech Republic (kotrbcik@fi.muni.cz)
}
\and Martin \v{S}koviera\thanks{Department of Computer Science Faculty of Mathematics, Physics and Informatics, Comenius University, Mlynsk\'a dolina,  842 48 Bratislava, Slovakia 
(skoviera@dcs.fmph.uniba.sk)}
}

\begin{document}

\maketitle

\begin{abstract}
The maximum genus $\gamma_M(G)$ of a graph $G$ is the largest
genus of an orientable surface into which $G$ has a cellular
embedding. Combinatorially, it coincides with the maximum
number of disjoint pairs of adjacent edges of $G$ whose removal
results in a connected spanning subgraph of~$G$. In this paper
we prove that removing pairs of adjacent edges from $G$
arbitrarily while retaining connectedness leads to at least
$\gamma_M(G)/2$ pairs of edges removed. This allows us to
describe a greedy algorithm for the maximum genus of a graph;
our algorithm returns an integer $k$ such that
$\gamma_M(G)/2\le k \le \gamma_M(G)$, providing a simple method
to efficiently approximate maximum genus. As a consequence
of our approach we obtain a $2$-approximate counterpart of
Xuong's combinatorial characterisation of maximum genus.
\end{abstract}



\section{Introduction}
The \textit{maximum genus} $\gamma_M(G)$ of a graph $G$ is the
maximum integer $g$ such that $G$ has a cellular embedding in
the orientable surface of genus $g$. A result of Duke
\cite{duke} implies that a graph $G$ has a cellular embedding
in the orientable surface of genus $g$ if and only if
$\gamma(G) \le g \le \gamma_M(G)$ where $\gamma(G)$ denotes the
(minimum) genus of $G$. The problem of determining the set of
genera of orientable surfaces upon which $G$ can be embedded
thus reduces to calculation of $\gamma(G)$ and $\gamma_M(G)$.

Computing the minimum genus of a graph is a notoriously
difficult problem, which is known to be NP-complete  even for
cubic graphs (see \cite{Th,T2}). Nevertheless, the minimum
genus can be calculated in linear time for graphs with bounded
genus or bounded treewidth by \cite{KMR:2008}. Moreover, for
graphs with fixed treewidth and bounded maximum degree
\cite{gross:2014} provides a polynomial-time algorithm
obtaining the complete genus distribution $\{g_i\}$ of
the graph $G$, where $g_i$ denotes the number of cellular
embeddings of $G$ into the orientable surface of genus $i$. For
graphs of bounded maximum degree \cite{CS:2013} has recently
proposed a polynomial-time algorithm constructing an embedding
with genus at most $O(\gamma(G)^{c_1}\log^{c_2} n)$ where $c_1$
and $c_2$ are constants. On the other hand, for every
$\epsilon>0$ and every function $f(n) = O(n^{1-\epsilon})$
there is no polynomial-time algorithm that constructs an
embedding of any graph $G$ with $n$ vertices into the surface
of genus at most $\gamma(G) + f(n)$ unless P $=$ NP (see
\cite{CKK,CKK2}).

For maximum genus, the situation is quite different, as maximum
genus admits a good (min-max) characterisation by Xuong's and
Nebesk\'y's theorems, see \cite{KOK,X} and \cite{KG,N},
respectively. From among these results the best known is
Xuong's theorem stating that $\gamma_M(G) = (\beta(G) - \min_T
\odd(G-E(T)))/2$, where $\beta(G)$ is the cycle rank of $G$,
$\odd(G-E(T))$ is the number of components of $G-E(T)$ with an
odd number of edges, and the minimum is taken over all spanning
trees $T$ of $G$. Building on these results, Furst
et~al.~\cite{FGM} and Glukhov \cite{G} independently devised
polynomial-time algorithms for determining the maximum genus of
an arbitrary graph. The algorithm of \cite{FGM} uses Xuong's
characterisation of maximum genus and exploits a reduction to
the linear matroid parity on an auxiliary graph; its running
time is bounded by $O(mn\Delta\log^6m)$, where $n,m$, and
$\Delta$ are the number of vertices, edges, and the maximum
degree of the graph, respectively. A matroidal structure is
also in the backgroung of the algorithm derived in~\cite{G},
albeit in a different way. Starting with any spanning tree $T$
of $G$, the algorithm  greedily finds a sequence of graphs
$F_i$ such that $T = F_0 \subseteq \cdots \subseteq F_n
\subseteq G$, $|E(F_{i+1}) - E(F_i)| = 2$, and $\gamma_M(F_i) =
i$ for all $i$, and $\gamma_M(F_n) = \gamma_M(G)$. The running
time of this algorithm is bounded by $O(m^6)$.

Although two polynomial-time algorithms for the maximum genus
problem are known, both are relatively complicated. It is
therefore desirable to have a simpler way to determine the
maximum genus, at least approximately. A greedy approximation
algorithm for the maximum genus of a graph was proposed by Chen
\cite{chen}. The algorithm has two main phases. First, it
modifies a given graph $G$ into a $3$-regular graph $H$ by
vertex splitting, chooses an arbitrary spanning tree $T$ of
$H$, and finds a  set $P$ of disjoint pairs  of adjacent edges
in $H-E(T)$ with the maximum possible size. Second, it
constructs a single-face embedding of $T\cup P$  and then
inserts the remaining edges into the embedding while trying to
raise the genus  as much as possible. A high-genus embedding of
$G$ in the same surface is then constructed by contracting the
edges created by vertex splitting. The algorithm constructs an
embedding of $G$ with genus at least $\gamma_M(G)/4$ and its
running time $O(m\log n)$ is dominated by the second phase,
that is, by operations on an embedded spanning subgraph of $H$.

In this paper we show that there is a much simpler way to
approximate maximum genus. Our algorithm repeatedly removes
arbitrary pairs of adjacent edges from $G$ while keeping the
graph connected. We prove that this simple idea leads to at
least $\gamma_M(G)/2$ pairs removed, providing an algorithm
that returns an integer $k$ such that $\gamma_M(G)/2 \le k \le
\gamma_M(G)$. This process can be implemented with running time
$O(m^2\log^2n/(n\log\log n))$. The algorithms developed in
\cite{chen} can then be used to efficiently construct an
embedding with genus $k$. Our result provides the first
method to approximate maximum genus that can be easily
implemented and improves the previous more complicated
algorithm of Chen~\cite{chen}, which can guarantee
embedding with genus only $\gamma_M(G)/4$.  Structurally,
our approach yields a natural $2$-approximate counterpart of
Xuong's theorem.

\section{Background}

In this section we present definitions and results that provide
the background for our algorithm.

Our terminology is standard and consistent with \cite{MT}.
By a graph we mean a finite undirected graph with loops and
parallel edges permitted. Throughout, all embeddings into
surfaces are cellular, forcing our graphs to be connected,
and the surfaces are orientable. For more details and the
necessary background  we refer the reader to \cite{GT}  or
\cite{MT}; a recent survey of maximum genus can be found in
\cite[Chapter~2]{topics}.

One of the earliest results on embeddings of graphs is the
following observation, which is sometimes called Ringeisen's
edge-addition lemma. Although it is implicit in \cite{NSW},
Ringeisen \cite{ringeisen:1972} was perhaps the first to draw
an explicit attention to it.

\begin{lemma}\label{lemma:edge-addition-technique}
Let $\Pi$ be an embedding of a connected graph $G$ and let $e$
be an edge not contained in $G$, but incident with vertices in
$G$
\begin{itemize}
\item[{\rm (i)}] If both ends of $e$ are inserted into the
    same face of $\Pi$, then this face splits into two
    faces of the extended embedding of $G+e$ and the genus
    does not change.
\item[{\rm (ii)}] If the ends of $e$ are inserted into two
    distinct faces of $\Pi$, then in the extended embedding
    of $G+e$ these faces are merged into one and the genus
    raises by one.
\end{itemize}
\end{lemma}

The next lemma, independently obtained in \cite{KOK},
\cite{jungerman:1978}, and \cite{X}, constitutes the
cornerstone of proofs of Xuong's theorem. It follows easily
from Lemma~\ref{lemma:edge-addition-technique}.

\begin{lemma}\label{lemma:adding-pairs-single-face}
Let $G$ be a connected graph and $\{e,f\}$ a pair of adjacent
edges not contained in $G$, but incident with vertices in $G$.
If $G$ has an embedding with a single face, then so does
$G\cup\{e,f\}$.
\end{lemma}

Recall that by Xuong's theorem $\gamma_M(G) = (\beta(G) -
\min_T \odd(G-E(T)))/2$, where $\odd(G-E(T))$ is the number of
components of the cotree $G-E(T)$ with an odd number of edges.
It is not difficult to see that every cotree component with an
even number of edges can be partitioned into pairs of adjacent
edges, and that every cotree component with an odd number of
edges can be partitioned into pairs of adjacent edges and one
unpaired edge. Therefore, any spanning tree $S$ minimising
$\odd(G-E(T))$ maximises the number of pairs in the above
partition of the cotree. The proof strategy of Xuong's theorem
can now be summarised as follows. First, embed $S$ in the
$2$-sphere arbitrarily. Then repeatedly apply
Lemma~\ref{lemma:adding-pairs-single-face} to pairs obtained
from the partition of the components of $G-E(S)$, each time
rising the genus by one. Finally, add the remaining edges.
Lemma~\ref{lemma:edge-addition-technique} guarantees that the
addition cannot lower the genus. The result of this process is
an embedding of $G$ with genus at least $(\beta(G) - \min_T
\odd(G-E(T)))/2$.

The fact that a spanning tree minimising $\odd(G-E(T))$
maximises the number of pairs of adjacent edges in the cotree
suggests a slightly different combinatorial characterisation of
maximum genus. It is due to  Khomenko et al.~\cite{KOK} and in
fact is older than Xuong's theorem itself.

\begin{theorem}\label{thm:KOK}
The maximum genus of a connected graph equals the maximum
number of disjoint pairs of adjacent edges whose removal leaves
a connected graph.
\end{theorem}

The following useful lemma, found for example in
\cite{CKK}, is an extension of
Lemma~\ref{lemma:adding-pairs-single-face} to embeddings with
more than one face. It can either be proved directly by using
Ringeisen's edge-adding technique or can be derived from
Xuongs's theorem. We may note in passing that this lemma was
used in \cite{CKK} to devise an algorithm that constructs an
embedding of genus $\gamma_M(G)-1$ whenever such an embedding
exists.

\begin{lemma}\label{lemma:adding-pairs}
Let $G$ be a connected graph and $\{e,f\}$ a pair of adjacent
edges not contained in $G$, but incident with vertices in $G$.
Then $\gamma_M(G\cup\{e,f\})\ge \gamma_M(G) + 1$.
\end{lemma}

 Our main observation is that
Lemma~\ref{lemma:adding-pairs} can be applied to sets of pairs
of adjacent edges which do not necessarily have the
maximum possible size. Indeed, if we find any $k$ pairs
of adjacent edges $(e_i,f_i)_{i=1}^{k}$ in a graph $G$ such
that $G-\bigcup_{i=1}^{k} \{e_i,f_i\}$ is connected, then by
Lemma~\ref{lemma:adding-pairs} we can assert that the maximum
genus of $G$ is at least $k$. This  suggests that identifying a
large number of pairs of adjacent edges whose removal leaves a
connected subgraph can be utilised to obtain a
simple approximation algorithm for the maximum genus.
Indeed, in the following section we show that choosing
the pairs of adjacent edges arbitrarily yields an
effective approximation of maximum genus.

\section{Algorithm}

In this section we present a greedy algorithm for finding at
least $\gamma_M(G)/2$ pairs of adjacent edges while the rest of
the graph remains connected. The idea is simple: if the removal
of a pair of adjacent edges does not disconnect the graph, then
we remove it. 

To prove that the set output by Greedy-Max-Genus Algorithm
always contains at least $\gamma_M(G)/2$ pairs of adjacent
edges we employ the following lemma, which can be easily proved
either using Xuong's theorem or directly from
Lemma~\ref{lemma:edge-addition-technique}.

\begin{lemma}\label{lemma:edge-removal}
Let $G$ be a connected graph and let $e$ be an arbitrary edge
of $G$ such that $G-e$ is connected. Then
$$\gamma_M(G)-1\le\gamma_M(G-e)\le \gamma_M(G).$$
\end{lemma}

\vbox{
\hrule
\vspace*{1mm}
Greedy-Max-Genus Algorithm
\hrule
\vspace*{1mm}
 \noindent 
 Input:  Connected graph $G$ \\
 Output:  Set $P$ of paiwise disjoint pairs of adjacent edges  of $G$ such that $G-P$ is \\ 
\hspace*{8mm} a connected spanning   subgraph of $G$\\
\hspace*{4mm} 1:  $H \leftarrow G$, $P \leftarrow \emptyset$\\
\hspace*{4mm} 2:  {\bf repeat }\\
\hspace*{4mm} 3:  \qquad choose adjacent edges $e, f$ from $H$   \\
\hspace*{4mm} 4:  \qquad {\bf if} $H - \{e,f\}$ is connected  \\
\hspace*{4mm} 5:  \qquad \qquad $H \leftarrow H - \{e,f\}$\\
\hspace*{4mm} 6:  \qquad \qquad $P \leftarrow P\cup (e,f)$ \\
\hspace*{4mm} 7:  {\bf until} all pairs of adjacent edges of $H$ have been tested\\
\hspace*{4mm} 8:  return  $P$
\vspace*{1mm}
\hrule
}

The final ingredient for our Greedy-Max-Genus Algorithm is the
following characterisation of graphs with maximum genus $0$.

\begin{theorem}\label{thm:mgZero}
The following statements are equivalent for every
connected graph $G$.
\begin{itemize}
\item[{\rm (i)}] $\gamma_M(G)=0$
\item[{\rm (ii)}] No two cycles of $G$ have a vertex in
    common.
\item[{\rm (iii)}] $G$ contains no pair of adjacent edges
    whose removal leaves a connected graph. 
\end{itemize}
\end{theorem}

The equivalence (i) $\Leftrightarrow$ (ii) in
Theorem~\ref{thm:mgZero} was first proved by Nordhaus
et~al.~in \cite{NRSW}. The equivalence (ii) $\Leftrightarrow$
(iii) is easy to see, nevertheless it is its appropriate
combination with Lemma~\ref{lemma:edge-removal} which yields
the desired performance guarantee for Greedy-Max-Genus
algorithm,  as shown in the following theorem.

\begin{theorem}\label{thm:alg}
For every connected graph $G$, the set of pairs output
by Greedy-Max-Genus Algorithm run on $G$ contains at least
$\gamma_M(G)/2$ pairs of adjacent edges.
\end{theorem}

\begin{proof}
Assume that the algorithm stops after the removal of $k$
disjoint pairs of adjacent edges from~$G$. For
$i\in\{0,1,\dots,k\}$ let $H_i$ denote the graph obtained from
$G$ by the removal of the the first $i$ pairs of edges.
By Lemma~\ref{lemma:edge-removal}, the removal of a single edge
from a graph can lower its maximum genus by at most one.
Therefore, the removal of two edges can lower the maximum genus
by at most two. It follows that $\gamma_M(H_i) \ge \gamma_M(G)
-2i$ for each $i$; in particular, $\gamma_M(H_k) \ge
\gamma_M(G) -2k$. From Theorem \ref{thm:mgZero} we get that
$\gamma_M(H_k) = 0$. By combining these expressions we get $2k
\ge \gamma_M(G)$, which yields $k\ge \gamma_M(G)/2$, as
desired.
\end{proof}

\begin{remark}
Let $n$ and $m$  denote the number of vertices and edges
of $G$, respectively, and let $k$ be the number of pairs of
adjacent edges produced by Greedy-Max-Genus Algorithm run on
$G$. An embedding of $G$ with genus at least $k$ can be
constructed from the set of pairs of adjacent edges in time
$O(n+k\log n)$,  see the proof of Theorem 4.5 in \cite{chen}
for details.
\end{remark}

Observe that any maximal set of pairs of adjacent edges
of $G$ whose removal from $G$ yields a connected graph can be
output by Greedy-Max-Genus Algorithm run on $G$. Hence, as a
corollary of Theorem~\ref{thm:alg} we obtain the following
2-approximate counterpart of Xuong's theorem.

\begin{theorem}\label{thm:approx}
Let $G$ be a connected graph and let $P$ be any inclusion-wise
maximal set of disjoint pairs of adjacent edges of $G$ whose
removal leaves a connected subgraph. Then
$|P|\ge\gamma_M(G)/2$.
\end{theorem}

The following example shows that the bound of Theorem
\ref{thm:approx} is tight and Greedy-Max-Genus Algorithm can
output the value $\gamma_M(G)/2$ for infinitely many graphs
$G$.

\begin{example}
\label{ex:1} 
Take the star $K_{1,2n}$ where $n$ is an arbitrary
positive integer, replace every edge with a~pair of parallel
edges, and add a loop to every vertex of degree $2$. Denote the
resulting graph by~$G_n$. Using Theorem~\ref{thm:KOK} it is
easy to see that $\gamma_M(G_n)= 2n$. Indeed, take a set
$P$ of $2n$ disjoint pairs of adjacent edges, each
consisting of a loop and one of its adjacent edges. Since the
edges of $G_n$ not in $P$ form a spanning tree, $P$ has maximum
size with respect to the property that $G-P$ is
connected. Thus $\gamma_M(G_n) = 2n$ by Theorem~\ref{thm:KOK}.
On the other hand, consider a set $P'$ of $n$ disjoint
pairs of adjacent edges that include only edges incident with
the central vertex. The removal of these pairs from $G_n$
leaves a spanning tree of $G_n$ with a loop attached to every
pendant vertex, so $P'$ is a maximal set of pairs for which
$G_n-P'$ is connected. Since $|P'|=n=\gamma_M(G_n)/2$, our
example confirms that the bound in Theorems~\ref{thm:alg}
and~\ref{thm:approx} is best possible.
\end{example}

Example \ref{ex:1} also implies that processing vertices  in
the decreasing order with respect to their degrees does not
necessarily  lead to a better performance of the algorithm.

\section{Implementation}

To implement the algorithm it is clearly sufficient to consider
all pairs of edges $\{e,f\}$ with a common end-vertex and test
whether removing the pair does not disconnect the graph.
The running time is thus
$O((\tau+\rho)\sum_{i=1}^n d_i^2)$, where $\tau$ is the time required to
test the connectivity, $\rho$ is the time required to update the
underlying data structure, and $d_i$ is the degree of the
$i$-th vertex.
If the input graph is simple, then $\sum_{i=1}^n d_i^2
= O(m^2/n)$ by \cite{caen:1998}. If  the input graph  is not
simple, we can preprocess it as follows. Let $P=\emptyset$.
From every set $F$ of pairwise parallel edges we repeatedly
remove pairs of edges and add them into $P$ until  $F$
contains at most two parallel edges. Simirarly, from every set
$F$ of loops incident with a single vertex we remove pairs of
loops and add them into $P$ until $F$ contains at most one
loop.
Let $G'$ denote the resulting graph.
 Finally, the set $P$ of pairs of adjacent edges is added
to the set $P'$ produced by Greedy-Max-Genus Algorithm on $G'$.
 It can be easily seen that in $G'$ the sum
of squares of the degrees is again in $O(m^2/n)$ and that the
preprocessing phase can be done in $O(m)$, where $m$ and $n$ is the  number of edges and vertices of the input graph. Regarding the
testing of connectivity, we have $\tau=O(m)$ for instance by using
DFS, in which case there is no need for additional updates of data
structure and thus $\rho=O(m)$. Therefore, we obtain an
implementation which reduces essentially to a series of
connectivity tests and has running time $O(m^3/n)$. Using the
dynamic graph algorithm for connectivity from \cite{WN:2013} it
is possible to support updates in $\rho  = O(\log^2 n/\log\log n)$
amortized time and queries in $\tau = O(\log n/\log\log n)$
worst-case time. This yields the total running time
$O(m^2\log^2n/(n\log\log n))$.

\section{Discussion}
We have presented an approximation algorithm for the maximum
genus problem that for any connected graph $G$ outputs an
integer $k$ such that $\gamma_M(G)/2 \le k \le \gamma_M(G)$.
This result shows that the classical ideas dating back to
Norhaus et al.~\cite{NRSW} and Ringeisen \cite{ringeisen:1972},
and to characterisations of maximum genus by Xuong \cite{X} and
Khomenko et al.~\cite{KOK},  can be used also for efficient
approximation of maximum genus. Our algorithm is much simpler
than both the $4$-approximation algorithm of Chen
\cite{chen} and the precise polynomial-time algorithms for the
maximum genus problem of Furst et al. \cite{FGM} and
Glukhov~\cite{G}, outperforms the existing
$4$-approximation algorithm \cite{chen}, and provides the first
approximation of maximum genus that can be easily implemented.
On the structural side, we have obtained a natural
$2$-approimation counterpart of Xuong's theorem.

\section*{Acknowledgement}
The first author was partially supported by Ministry of
Education, Youth, and Sport of Czech Republic, Project
No.~CZ.1.07/\-2.3.00/\-30.0009. The second author was
partially supported by VEGA grant 1/0474/15. The authors would
like to thank Rastislav Kr\'alovi\v{c} and Jana
Vi\v{s}\v{n}ovsk\'a for reading preliminary versions of this
paper and making useful suggestions.

\scriptsize
\bibliographystyle{plain}
\bibliography{bibl}
\end{document}